\documentclass[11pt,twoside,a4paper]{amsart}
\usepackage{amsmath,amsthm,amsfonts,amssymb,latexsym}
\usepackage{url}

\title[Curvature restrictions for Levi-flats in $\C\PP^2$]
{Curvature restrictions for Levi-flat real hypersurfaces in complex projective planes}

\author{Masanori Adachi}
\address[M. Adachi]{Center~for~Geometry~and~its~Applications, Pohang~University~of~Science~and~Technology, Pohang 790-784, Republic of Korea
\& 
Graduate~School~of~Mathematics, Nagoya~University, Nagoya 464-8602, Japan}
\email{adachi@postech.ac.kr}

\author{Judith Brinkschulte}
\address[J. Brinkschulte]{Universit\"at Leipzig, Mathematisches Institut, PF 100920, D-04009 Leipzig, Germany}
\email{brinkschulte@math.uni-leipzig.de}

\date{\today}

\thanks{The first author is partially supported by 
an NRF grant 2011-0030044 (SRC-GAIA) of the Ministry of Education, the Republic of Korea, and 
a JSPS Grant-in-Aid for Young Scientists (B) 26800057.
}

\numberwithin{equation}{section}

\newcommand\C{\mathbb{C}}  
\newcommand\R{\mathbb{R}}
\newcommand\N{\mathbb{N}}

\newcommand\PP{\mathbb{P}}
\newcommand\Ric{\mathrm{Ric}}
\newcommand\Ker{\mathrm{Ker}}

\newcommand{\base}[1]{\frac{\pa}{\pa #1}}

\newcommand{\pa}{\partial}
\newcommand{\opa}{\overline\pa}
\newcommand{\ol}{\overline }

\newtheorem*{MainTheorem}{Main Theorem}
\newtheorem*{Theorem*}{Theorem}

\newtheorem{Theorem}{Theorem}[section]
\newtheorem{Proposition}[Theorem]{Proposition}

\newtheorem{Lemma}[Theorem]{Lemma}
\newtheorem{Corollary}[Theorem]{Corollary}

\theoremstyle{remark}
\newtheorem{Remark}[Theorem]{Remark}

\begin{document}

\maketitle

\begin{abstract} 
We study curvature restrictions of Levi-flat real hypersurfaces in complex projective planes, 
whose existence is in question. 
We focus on its totally real Ricci curvature, the Ricci curvature of the real hypersurface in the direction of the Reeb vector field, and show that it cannot be greater than $-4$ along a Levi-flat real hypersurface.
We rely on a finiteness theorem for the space of square integrable holomorphic 2-forms on the complement of the Levi-flat real hypersurface,
where the curvature plays the role of the size of the infinitesimal holonomy of its Levi foliation. 
\end{abstract}

\section{Introduction}

The past decades, the non-existence conjecture of a smooth closed Levi-flat real hypersurface in the complex projective spaces $\C\PP^n$ $(n \geq 2)$
has been intensively investigated by foliators, complex analysts, and differential geometers.
This conjecture first appeared in the papers \cite{CLS} and \cite{C} devoted to the study of minimal sets of holomorphic foliations on $\C\PP^n$ 
 and  has been affirmatively proved for $n > 2$ by Lins Neto \cite{L} in the real analytic case
and by Siu \cite{Si} in the smooth case. 
There have been papers that announced proofs of the non-existence for $n=2$, however, 
they all are considered to contain serious gaps (cf. \cite{IM}) and the case $n=2$ remains open.

The following partial result by Bejancu and Deshmukh uses a differential-geometric approach to restrict a certain curvature of the Levi-flat real hypersurface:

\begin{Theorem*}[\cite{BD}]
\label{bejancu-deshmukh}
Let $M$ be an oriented $\mathcal{C}^\infty$-smooth closed Levi-flat real hypersurface in $\C\PP^n$ $(n \geq 2)$
equipped with the Fubini--Study metric. 
Denote by $\nu$ the unit normal vector field of $M \subset \C\PP^n$ and set $\xi = -J\nu$, 
where $J$ denotes the complex structure of $\C\PP^n$. 
Then $\Ric^M(\xi, \xi)$ cannot be $\geq 0$ everywhere on $M$.
\end{Theorem*}

The Ricci curvature $\Ric^M(\xi, \xi)$ is referred to as \emph{the totally real Ricci curvature} of the real hypersurface $M$.
The aim of this paper is to improve the curvature restriction in the theorem of Bejancu and Deshmukh. Also, our proof is completely complex-analytic.

Our main theorem is stated as follows:
\begin{MainTheorem}
Let $M$ be an oriented $\mathcal{C}^2$-smooth closed Levi-flat real hypersurface in $\C\PP^2$ equipped with the Fubini--Study metric. 
Then, the totally real Ricci curvature $\Ric^M(\xi, \xi)$ cannot be $> -4$ everywhere on $M$. 
\end{MainTheorem}

\medskip

The idea of the proof is as follows. 
The key ingredient is a leafwise $(1,0)$-form $\alpha$ on the Levi-flat real hypersurface $M$,
which measures the size of the infinitesimal holonomy of the Levi foliation of $M$ (See \S\ref{subsect:alpha} for its definition). 
We will find out that the restriction on the totally real Ricci curvature is equivalent to an upper bound on the norm of $\alpha$ 
by observing an adjunction-type equality (Proposition \ref{local}).
Then, exploiting an integral formula (Theorem \ref{integral}) originating from a paper of Griffiths, 
we will prove the finite dimensionality of the space of $L^2$ holomorphic 2-forms on domains with Levi-flat boundary in $\C\PP^2$ (Corollary \ref{finite})
under the upper bound on the norm of $\alpha$. 
This finite dimensionality is however a contradiction, because the space should be infinite dimensional
(Proposition \ref{separate}).

\medskip

The organization of this paper is as follows. 
In \S\ref{sect:prelim}, we explain the basic notions and our conventions pertaining to the local geometry of Levi-flat real hypersurfaces. 
The form $\alpha$ is defined in this section.
The connection between the form $\alpha$ and the totally real Ricci curvature is explained in \S\ref{sect:adjunction} via Gauss' equation. 
In \S\ref{sect:integral}, we explain and prove some integral formula \`a la Griffiths. 
We exploit this formula in \S\ref{sect:finiteness} to study the finite dimensionality of the space of $L^2$ holomorphic sections of 
negative holomorphic line bundles $\mathcal{O}_{\C\PP^2}(-m)$ over domains with Levi-flat boundary in $\C\PP^2$. 
In \S\ref{sect:infiniteness}, we give a proof of the infinite dimensionality of the $L^2$ canonical sections on pseudoconvex domains in complex projective spaces and complete the proof of the Main theorem.
In Appendix \ref{sect:takeuchi}, we revisit Takeuchi's inequality from the viewpoint of \S\ref{sect:adjunction} and 
give a remark on a general restriction on the totally real Ricci curvature of Levi-flat real hypersurfaces in K\"ahler surfaces.

\subsection*{Acknowledgements}
The authors are grateful to K. Matsumoto for explaining her recent work to the first author and mentioning a result in \cite{O}, which leads us to Proposition \ref{takeuchi}.
We also thank A. Iordan for valuable discussions and the referee for suggestions improving the quality of the manuscript.

\section{Preliminaries for local arguments}
\label{sect:prelim}

In this section, we collect basic notions and our conventions 
pertaining to the local geometry of Levi-flat real hypersurfaces. 
We restrict ourselves to the case of K\"ahler surfaces. 

\subsection{K\"ahler metric and the bisectional curvature}
Let $(X, J_X)$ be a complex surface. 
The complex structure $J_X: TX \to TX$ allows us to regard the real tangent bundle $TX$ as a $\C$-vector bundle. 
We identify $TX$ with the holomorphic tangent bundle $T^{1,0}X$ as $\C$-vector bundles by
\[
\base{x_j} \longmapsto \base{z_j} = \frac{1}{2}\left(\base{x_j} - i\base{y_j}\right)
\]
where $(z_1, z_2)$ denotes a local coordinate and we write $z_j = x_j + i y_j$. 

Let $g: TX \times TX \to \R$ be a $J_X$-invariant Riemannian metric of $X$. 
We consider its sequilinear extension on $\C \otimes TX \times \C \otimes TX$ and 
obtain a hermitian metric of $T^{1,0}X$, namely, 
\[
g\left(\base{z_j}, \base{z_k}\right) = \frac{1}{2}\left(g\left(\base{x_j}, \base{x_k}\right) + ig\left(\base{x_j}, \base{y_k}\right)\right).
\]
The metric $g$ is said to be K\"ahler if its fundamental form 
\[
\omega = 2i \sum_{j,k=1}^2 g\left(\base{z_j}, \base{z_k}\right) dz_j \wedge d\ol{z}_k
\]
is a closed form. Our volume form is $dV_\omega = \omega \wedge \omega/8$.

We denote by $\nabla$ the Levi-Civita connection determined by the Riemannian metric $g$. 
It is well-known that $\nabla$ coincides with the canonical connection of the hermitian metric $g$.
Our convention of the curvature tensor of $\nabla$ is
\[
R(v_1, v_2) v_3 := \nabla_{v_1} (\nabla_{\widetilde{v_2}} \widetilde{v_3}) - \nabla_{v_2} (\nabla_{\widetilde{v_1}} \widetilde{v_3}) - \nabla_{[\widetilde{v_1}, \widetilde{v_2}]} \widetilde{v_3}
\]
for $v_j \in T_p X$ where $\widetilde{v_j}$ are extensions of $v_j$ to vector fields respectively. 
Given two $J_X$-invariant planes $\sigma_1, \sigma_2 \subset T_pX$,
we define the bisectional curvature by 
\begin{align*}
H(\sigma_1, \sigma_2) & := g(R(v_1, Jv_1)Jv_2, v_2) \\
&= g(R(v_1, v_2)v_2, v_1) + g(R(v_1, Jv_2)Jv_2, v_1)
\end{align*}
where $v_j$ are unit vectors in $\sigma_j$ respectively.

Our main example is the complex projective plane $X = \C\PP^2$ endowed with the Fubini--Study metric 
$g = g_{FS}$, whose fundamental form is given by 
\[
\omega_{FS} = i \pa\opa \log (1 + |z_1|^2 + |z_2|^2)
\]
in non-homogeneous coordinate system $(z_1, z_2)$. 
Note that $\omega_{FS}$ can be regarded as the Chern curvature $i\Theta(\mathcal{O}(1))$ of the hyperplane line bundle $\mathcal{O}(1)$
with the hermitian metric induced from the standard Euclidean metric of $\C^3$.
In our conventions, the bisectional curvature of $(\C\PP^2, \omega_{FS})$ is given by (cf. \cite{GK}) 
\[
H(\sigma_1, \sigma_2) = 2(1 + g(v_1, v_2)^2 + g(v_1, Jv_2)^2).
\]

\subsection{Real hypersurfaces and their holomorphic normal bundle}
\label{subsect:normal}
Let $M$ be an oriented $\mathcal{C}^2$-smooth closed real hypersurface without boundary in $X$. 
A \emph{defining function} $\rho$ of $M$ is a $\mathcal{C}^2$-smooth real-valued function defined on a neighborhood $U$ of $M$ 
expressing $M = \lbrace z \in U \mid \rho(z) = 0 \rbrace$ as the preimage of a regular value $0$.  
We always assume that $M$ is oriented as the boundary of $\{ z \in U \mid \rho(z) < 0 \}$ by reversing the sign of $\rho$ if necessary.
Later in \S\ref{sect:adjunction}, \S\ref{sect:finiteness} and Appendix \ref{sect:takeuchi} 
we will choose $\rho$ as the signed boundary distance function to $M$ with respect to the Riemannian metric $g$.

Along the real hypersurface $M$, we consider two smooth $J_X$-invariant plane fields: 
$\sigma_T := TM \cap J_X TM$, the unique $J_X$-invariant subbundle of $TM$, 
and $\sigma_N$, the orthonormal complement of $\sigma_T$ in $TX$ with respect to $g$. 
Via the $\C$-vector bundle isomorphism $TX \simeq T^{1,0}X$, 
we identify $\sigma_T$ with the holomorphic tangent bundle of $M$, $T^{1,0}M := \Ker \pa\rho \subset T^{1,0}X|M$, 
and $\sigma_N$ with the holomorphic normal bundle of $M$, $N^{1,0}_M := (T^{1,0} X|M)/T^{1,0}M$. 

We fix a global orthonormal frame $\{ \xi, \nu \}$ of $\sigma_N$ with respect to $g$, 
so that $\nu$ is an outward normal vector of $M \subset X$ and $\nu = J\xi$. 
The vector field $\xi$, tangent to $M$, is often referred to as a Reeb vector field.
We write the normal derivative $h_\rho := \nu\rho: M \to \R_{>0}$. 
Note that $h^2_\rho$ defines a $\mathcal{C}^1$-smooth hermitian metric of $N^{1,0}_M$; 
we can measure the squared norm of a  vector $v =  a\xi + b\nu \in \sigma_N \simeq N^{1,0}_M$ by
$|v|^2_\rho := h^2_\rho |a + ib|^2$.

\begin{Remark}
We are working with a fixed K\"ahler metric for simplicity, though note that 
$N^{1,0}_M$ and the hermitian metric induced from $h^2_\rho$ are independent of the choice of the hermitian metric $g$
whereas $\sigma_N$ and $\{\xi, \nu\}$ depend on $g$. 
\end{Remark}

\subsection{Levi-flat and a distinguished parametrization}

The \emph{Levi-form} of $M$ (with respect to $\rho$) is the restriction of the quadratic form 
obtained by $i\pa\opa \rho$ to the holomorphic tangent bundle $T^{1,0}M$. 
The real hypersurface $M$ is said to be \emph{Levi-flat} if the Levi-form of $M$ vanishes identically on $M$. 
It is easy to see that this definition does not depend on the choice of $\rho$. 

It follows from Frobenius' theorem that $M$ is Levi-flat if and only if 
$M$ has a foliation by complex hypersurfaces of $X$. 
The foliation is called the \emph{Levi foliation} $\mathcal{F}$ of $M$. 
The typical example of Levi-flat real hypersurfaces is $\C \times \R \subset \C^2$,
where the Levi foliation is given by $\{ \C \times \{t\} \}_{t \in \R}$. 

It is well-known that any real-analytic Levi-flat real hypersurface is locally identified with this typical example. 
For non real-analytic Levi-flat real hypersurfaces, although its local structure is not unique, 
still a sort of normal coordinate system is available. 
Suppose $M$ is Levi-flat. 
By standard arguments (cf. \cite{A}), we can choose a holomorphic chart $(z_1, z_2)$ of $X$ 
and a local parametrization $\varphi$ of $\mathcal{F}$ around any point $p \in M$ so as to satisfy the following conditions:
\begin{enumerate}
\item The parametrization $\varphi(\zeta,t): \C \times \R \supset V \to \varphi(V) \subset M$ is 
a $\mathcal{C}^2$-smooth orientation-preserving diffeomorphism with $\varphi(0,0) = p$ and holomorphic in $z$.
(We always assume these conditions for parametrizations of Levi foliations.)
\item The parametrization $\varphi$ is in the form of $\varphi(\zeta, t) =  (\zeta, w(\zeta, t))$ in the coordinate system $(z_1, z_2)$.
\item The parametrization behaves along the leaf passing through $p$ in such a way that 
\[
w(\zeta, 0) \equiv 0 
\quad \text{and} \quad 
\frac{\pa w}{\pa t}(\zeta, 0) \equiv 1,
\]
namely,
\[
\varphi_*\left(\base{t}\right)_{(\zeta,0)} = \left(\base{x_2}\right)_{(\zeta, 0)}.
\]

\item The coordinate system $(z_1, z_2)$ is normalized at $p$ with respect to $g$,
namely,  its fundamental form $\omega$ satisfies $\omega = i (dz_1 \wedge d\ol{z_1} + dz_2 \wedge d\ol{z_2})$ at $p$. 
\end{enumerate}
We refer to such a parametrization $\varphi(\zeta,t)$ of $M$ in the coordinate system $(z_1, z_2)$ 
as a \emph{distinguished parametrization} around $p \in M$ in this paper. 

\subsection{The form $\alpha$}
\label{subsect:alpha}

Let us define the key object of this paper. 

Suppose $M$ is Levi-flat and a $\mathcal{C}^1$-smooth hermitian metric $h^2$ of $N^{1,0}_M$ is given. 
Consider a parametrization of the Levi foliation $\mathcal{F}$, not necessarily a distinguished parametrization, 
say $\varphi(\zeta,t): \C \times \R \subset V \to M$. 
This parametrization gives us a local trivialization of $N^{1,0}_M$ over $\varphi(V)$ 
since $\base{t}$ induces a local section of $N^{1,0}_M$.
We denote by $h^2_\varphi$ the local weight function of $h^2$ in this local trivialization. 
Using these notations, we define a continuous leafwise $(1,0)$-form defined on $\varphi(V)$ by 
\[
\alpha := \frac{\pa \log h_\varphi}{\pa \zeta} d\zeta. 
\]
One easily sees that $\alpha$ is well-defined on $M$. 
The point is that $N^{1,0}_M$ becomes a leafwise flat line bundle, namely all the transition functions are leafwise constant 
if we equip $N^{1,0}_M$ with the local trivializations given by parametrizations of the Levi foliation.

\begin{Remark}
We can associate a transversal measure $\mu = hdt$ of $\mathcal{F}$ from the hermitian metric $h^2$ of $N^{1,0}_M$. 
The form $\alpha$ measures the infinitesimal holonomy with respect to this transversal measure $\mu$.
In particular, $\mu$ is holonomy invariant measure if and only if $\alpha \equiv 0$. 
The form  $\alpha$ is essentially the modular form of $\mu$ in the context of foliation, 
which is useful in the study of the $\opa$-Neumann problem on weakly pseudoconvex domains (cf. \cite{St}).
\end{Remark}

The following Lemma will be used in \S\ref{sect:finiteness}. 
We can describe $\alpha$ in terms of defining function of $M$ 
when the hermitian metric of $N^{1,0}_M$ is induced from a defining function of $M$.

\begin{Lemma}
\label{alpha}
Suppose $M$ is Levi-flat. Let $\rho$ be a defining function of $M$ and consider $\alpha$ induced from $h_\rho^2$.
Then, $\alpha$ is characterized as a continuous leafwise $(1,0)$-form on $M$ satisfying
\[
\pa\opa \rho = \alpha \wedge \opa\rho + \pa\rho \wedge \ol{\alpha} \pmod{\mathcal{C}^0(M) \pa\rho \wedge \opa\rho}.
\]
\end{Lemma}

\begin{proof}
First see that this equality makes sense modulo $\pa\rho \wedge \opa\rho$. 
This is because a leafwise $(1,0)$-form $\alpha$ is a linear functional on $T^{1,0}M = \Ker \pa\rho$ 
and its extension on $T^{1,0}X$ is unique modulo $\pa\rho$. 

Now we regard the equality as an equation on $\alpha$ and solve this. 
Take a point $p \in M$ and a distinguished parametrization $\varphi(\zeta,t): V \to M$ in $(z_1, z_2)$. 
Then, we have at $p$, as an alternating 2-form,
\begin{align*}
(\pa\opa\rho)_p\left(\base{z_1}, \base{\ol{z_2}}\right)
&= \frac{\pa^2 \rho}{\pa z_1 \pa \ol{z_2}} (0,0)
= \frac{i}{2}\frac{\pa^2 \rho}{\pa z_1 \pa y_2} (0,0)\\
&
= \frac{i}{2}\frac{\pa}{\pa \zeta}\left(\frac{\pa \rho}{\pa y_2}\right) (0,0).
\end{align*}
Note that $\varphi$ gives the standard embedding $\C \times \{0\} \cap V \subset \C^2$ and
we are allowed to confuse $\zeta$ and $z_1$ on this image. 
Our local weight function $(h_\rho)_\varphi$ is just given by $\frac{\pa \rho}{\pa y_2}$,
therefore, we have
\[
(\pa\opa\rho)_p\left(\base{z_1}, \base{\ol{z_2}}\right)
= \frac{i}{2}\frac{\pa (h_\rho)_\varphi}{\pa \zeta} (0,0).
\]
Similarly, 
\begin{align*}
(\alpha \wedge \opa\rho)_p \left(\base{z_1}, \base{\ol{z_2}}\right)
&= \alpha_p\left(\base{z_1}\right) (\opa\rho)_p\left(\base{\ol{z_2}}\right)
= \alpha_p\left(\base{z_1}\right)\frac{\pa \rho}{\pa \ol{z_2}} (0,0)\\
&= \alpha_p\left(\base{\zeta}\right) \cdot \frac{i}{2} (h_\rho)_\varphi (0,0).
\end{align*}
Hence, $\alpha$ should agree with the one given above. 

This $\alpha$ actually gives the solution since we can check the equality by looking at 
\[
(\pa\opa\rho)_p\left(\base{z_1}, \base{\ol{z_1}}\right) = 0, 
\quad (\pa\opa\rho)_p\left(\base{z_2}, \base{\ol{z_1}}\right) = \ol{(\pa\opa\rho)_p\left(\base{z_1}, \base{\ol{z_2}}\right)}.
\]
Since we want to compute $\pa\opa\rho$ only modulo $\pa\rho\wedge\opa\rho$, 
 we can ignore the contribution coming from $(\pa\opa\rho)_p\left(\base{z_2}, \base{\ol{z_2}}\right)$.
\end{proof}

\section{An adjunction-type equality}
\label{sect:adjunction}

In this section, we relate the totally real Ricci curvature, which we are going to estimate, 
with the form $\alpha$ via the Gauss' equation. 
Let $M$ be an oriented $\mathcal{C}^2$-smooth Levi-flat real hypersurface without boundary in a K\"ahler surface $(X, J_X, g)$.
We restrict the Riemannian metric $g$ on $M$ and 
denote by $\nabla^M$, $R^M$, $\Ric^M$ its Levi-Civita connection, curvature tenor, Ricci tensor respectively.
We will compare the bisectional curvature of $X$ and the totally real Ricci curvature of $M$ and see that 
their difference is exactly the squared norm of $\alpha$ induced from the signed boundary distance function. 

We consider near $M$ the signed boundary distance function $\delta$ with respect to the Riemannian metric $g$, namely, 
\[
\delta(p) = \pm \inf_{q \in M} \mathrm{dist}_g(p, q) 
\]
where we choose the sign so that $\delta$ becomes a defining function of $M$. 
It is well-known that $\delta$ is actually of $\mathcal{C}^2$-smooth near $M$.
Using this particular defining function $\delta$ of $M$, we induce a hermitian metric, simply denoted by $h^2$, 
on the holomorphic normal bundle $N^{1,0}_M$ as described in \S\ref{subsect:normal}, 
and consider the form $\alpha$ with respect to this $h^2$. 

Under this setting, we compute the difference in terms of $\alpha$ as follows:

\begin{Proposition}
\label{local}
The following equality holds:
\[
H\left(\sigma_T, \sigma_N \right) - \Ric^M(\xi, \xi) =  4i\alpha \wedge \ol{\alpha}/\omega
\]
where $\omega$ is the fundamental form of $g$ and the ratio of $i\alpha \wedge \ol{\alpha}$ and $\omega$ is taken as quadratic forms on $T^{1,0}M$.
\end{Proposition}

The main ingredient of the proof is Gauss' equation: for any real hypersurface $M$ in a Riemannian manifold, we have
\begin{align*}
& g(R(v_1, v_2)v_3, v_4) - g(R^M(v_1, v_2)v_3, v_4) \\
& = g(Av_1, v_3) g(Av_2, v_4) - g(Av_2, v_3) g(Av_1, v_4)
\end{align*}
where $v_j \in TM$ and $A$ denotes the shape operator of $M \subset X$, namely, for $v \in T_pM$, we let
\[
A v := -\nabla_v \nu \in T_pM.
\]

\begin{proof}[Proof of Proposition \ref{local}]
Fix a point $p \in M$ and take a distinguished parametrization around $p$, say $\varphi(\zeta,t): V \to M$ in $(z_1, z_2)$. 
Applying Gauss' equation at $p$ for 
\[
v_1 = v_4 = \left(\base{x_1}\right)_p, \quad v_2 = v_3 = \left(\base{x_2}\right)_p = \xi_p
\]
and 
\[
v_1 = v_4 = \left(\base{y_1}\right)_p, \quad v_2 = v_3 = \left(\base{x_2}\right)_p = \xi_p
\] 
respectively and adding resulting two equalities, we have at $p$
\begin{align}
H\left(\sigma_T, \sigma_N \right) - \Ric^M \left(\xi, \xi \right) \label{gauss} 
& = g(A\base{x_1}, \xi)^2 + g(A\base{y_1}, \xi)^2 \\
& - g(A\xi, \xi) \left(g(A\base{x_1}, \base{x_1}) + g(A\base{y_1}, \base{y_1})\right).  \nonumber
\end{align}
Note that the totally real Ricci curvature at $p$ is by its definition
\[
\Ric^M(\xi, \xi) = g(R\left(\base{x_1}, \xi\right)\xi, \base{x_1}) +  g(R\left(\base{y_1}, \xi\right)\xi, \base{y_1}).
\]

We observe that the last term in (\ref{gauss}) is zero. 
This is because any complex submanifold in any K\"ahler surface is minimal with respect to the K\"ahler metric, 
hence, the trace of the shape operator restricted on the tangent space of a complex submanifold is always zero.
We therefore have
\[
H\left(\sigma_T, \sigma_N \right) - \Ric^M \left(\xi, \xi \right) = g(A\base{x_1}, \xi)^2 + g(A\base{y_1}, \xi)^2.  \\
\]

The rest of the proof is to show the equality
\[
g(A\base{x_1}, \xi)^2 + g(A\base{y_1}, \xi)^2 = 4i\alpha \wedge \ol{\alpha}/\omega
\]
by direct computation. First we compute its first term at $p$
\[
g(A\base{x_1}, \xi)  = -g(\nabla_{\base{x_1}} \nu, \xi).
\] 
Our normal vector field $\nu$ is expressed  as
\[
\nu = \sqrt{ \frac{g_{x_1 x_1}}{g_{x_1x_1}g_{y_2y_2} - (g_{x_1y_2}^2 + g_{y_1y_2}^2)} } 
\left( \base{y_2} - \frac{g_{x_1y_2}}{g_{x_1x_1}}\base{x_1} - \frac{g_{y_1y_2}}{g_{y_1y_1}}\base{y_1} \right). 
\]
on $\C \times \{0\} \cap V$ where we use notations
\[
g_{x_j x_k} := g\left(\base{x_j}, \base{x_k}\right), 
g_{x_j y_k} := g\left(\base{x_j}, \base{y_k}\right), 
g_{y_j y_k} := g\left(\base{y_j}, \base{y_k}\right)
\]
for short. Using the explicit expression of the Christoffel symbol of the Levi-Civita connection 
and the K\"ahlerity of $g$, we have at $p$
\begin{align*}
g(A\base{x_1}, \xi) 
& = -g(\nabla_{\base{x_1}} \base{y_2}, \base{x_2})\\
&= -\frac{1}{2} \left(\base{x_1} g_{y_2 x_2} + \base{y_2} g_{x_1 x_2} - \base{x_2} g_{x_1 y_2}\right) \\
&= -\frac{1}{2} \left(\base{y_2} g_{x_1 x_2} - \base{x_2} g_{x_1 y_2}\right) \\
&= -\frac{1}{2} \left(\base{y_1} g_{x_2 x_2} \right).
\end{align*}
We can compute the second term at $p$ in the same way: 
\begin{align*}
g(A\base{y_1}, \xi) 
&= \frac{1}{2} \left(\base{x_1} g_{x_2 x_2} \right).
\end{align*}

On the other hand, we have on $\C \times \{0\} \cap V$, 
\begin{equation}
\label{adjunction-metric}
h_\varphi = \frac{\pa \delta}{\pa y_2} = g(\nu, \base{y_2}) = \sqrt{g_{x_2x_2} - \frac{g_{x_1y_2}^2 + g_{y_1y_2}^2}{g_{x_1x_1}}}.
\end{equation}
and we have at $p$
\[
i\alpha \wedge \ol{\alpha}/\omega
= \left|\base{z_1} \log h_\varphi\right|^2 
= \frac{1}{16} \left( \left(\base{x_1} g_{x_2x_2}\right)^2 + \left(\base{y_1} g_{x_2x_2}\right)^2 \right).
\]
This completes the proof. 
\end{proof}

\begin{Corollary}
Let $M$ be an oriented $\mathcal{C}^2$-smooth Levi-flat real hypersurface without boundary in a K\"ahler surface $X$. 
Then its totally real Ricci curvature satisfies 
\[
\Ric^M(\xi, \xi) \leq H\left(\sigma_T, \sigma_N \right).
\]
In particular, when $X$ is $\C\PP^2$ equipped with the Fubini--Study metric,
\[
\Ric^M(\xi, \xi) \leq 2.
\]
\end{Corollary}

\begin{Remark}
The induced metric $h_\varphi$ in (\ref{adjunction-metric}) is exactly the metric induced from 
the adjunction formula for Levi-flat real hypersurfaces (cf. \cite{Der}), 
\[
K_X|M \otimes N^{1,0}_M = (T^{1,0}M)^*. 
\]
Here we equip $K_X|M$ and $(T^{1,0}M)^*$ with the hermitian metrics induced from the given K\"ahler metric $g$. 
From this viewpoint we will revisit Takeuchi's inequality \cite{T} in Appendix \ref{sect:takeuchi}. 
\end{Remark}

\section{An integral formula}
\label{sect:integral}

In this section, we will prove an estimate for $L^2$-norm of smooth sections of hermitian holomorphic line bundles 
over {pseudoconcave domains in K\"ahler surfaces. 
This will be used in the next section to show a finiteness theorem for holomorphic sections 
of negative holomorphic line bundles over domains with Levi-flat boundary in $\C\PP^2$. 

As in the previous sections, let $X$ be a complex surface equipped with a K\"ahler metric $g$ and denote its fundamental form by $\omega$. 
We consider a relatively compact domain $\Omega \Subset X$ with $\mathcal{C}^2$-smooth boundary $M$. 
Later we will assume $M$ to be pseudoconcave in \S\ref{subsect:pseudoconcave} or Levi-flat in \S\ref{sect:finiteness}. 

Take a defining function $\rho$ of $M$ that is extended on $\ol{\Omega}$. 
We normalize our defining function so as to satisfy $\vert d\rho\vert_g =1$ on $M$; 
this is always possible replacing $\rho$ by $\rho/\vert d \rho\vert_g$ near $M$ and 
using a partition of unity argument.
Later in \S\ref{sect:finiteness} we will choose $\rho$ as the signed boundary distance function to $M$. 
We will denote by $dv_{M}$ the area element of $M$ with respect to the restriction of the Riemannian metric $g$ on $M$.
\\

\subsection{An integral formula for functions over complex manifolds}
\label{subsect:griffiths}
The main ingredient of our $L^2$-estimate is an integral formula for functions, which has already appeared in Griffiths' paper \cite{G}. 

Choose locally an orthonormal frame $\omega^1,\omega^2$ of $(1,0)$-forms with dual frame $L_1, L_2$ of $T^{1,0}X$ with respect to 
the hermitian metric $g$. On $M$, we also require that $\omega^2= \sqrt{2} \pa \rho$.
Then, $L_1$ gives a local trivialization of the holomorphic tangent bundle $T^{1,0}M$ and 
the Levi-form of $M$ can be identified with a scalar function, say $\ell_{\rho}(p) = \pa\opa \rho(p)(L_1,\ol L_1)$ for $p \in M$.
It is easily seen that $\ell_{\rho}$ is independent of the choice of the orthonormal frame. \\

We now define the nonnegative $(1,1)$-form $\omega_\rho$ by
$$\omega_\rho = 2i\pa\rho\wedge\opa\rho.$$
What is important is that the form $\omega_\rho$ has essentially the same properties as the form $\Omega_\tau$ in the paper of Griffiths \cite{G} for $n=2$.
One can then prove the following integral formula (see also \cite[p.\ 433]{G}):

\begin{Theorem}
\label{integral}
For any function $g\in\mathcal{C}^\infty(\ol \Omega)$ one has
$$\int_\Omega (i\pa\opa g)\wedge\omega_\rho - \int_\Omega ig\pa\opa\omega_\rho = \int_M g\cdot \ell_\rho 4dv_M.$$
\end{Theorem}

\begin{proof}
First of all it is easy to check that one has 
\begin{equation}  \label{lemma}
i(\pa-\opa)\omega_\rho \mid M = 8 \ell_\rho dv_{M}.
\end{equation} 
Indeed,   
\begin{equation}  \label{a}
i(\pa-\opa)\omega_\rho = -2(\pa-\opa) (\pa\rho\wedge\opa\rho) = 2\pa\opa\rho \wedge(\pa\rho-\opa\rho).
\end{equation}
On the other hand, we have 
\begin{equation}  \label{b}
dv_{M} = \ast d\rho \mid M = \frac{1}{4} \omega_1\wedge\ol\omega_1 \wedge(\pa\rho-\opa\rho).
\end{equation}
Since $\pa\rho\wedge\opa\rho =0$ on $M$, (\ref{a}) and (\ref{b}) imply (\ref{lemma}).\\

Now (\ref{lemma}) and Stokes' theorem imply
\begin{equation}  \label{c}
\int_{M} g \cdot \ell_\rho 4dv_{M} = \frac{1}{2} \int_{M} g i(\pa-\opa) \omega_\rho = \frac{1}{2} \int_\Omega d\lbrace g i(\pa-\opa)\omega_\rho\rbrace.
\end{equation}
But for bidegree reasons
\begin{eqnarray*} 
d\lbrace gi(\pa-\opa)\omega_\rho\rbrace 
& = & -2gi\pa\opa\omega_\rho + dg\wedge i(\pa-\opa)\omega_\rho\\
& = & -2gi\pa\opa\omega_\rho + i\left( - \pa g \wedge \opa \omega_\rho + \opa g \wedge \pa \omega_\rho \right)\\
& = & -2gi\pa\opa\omega_\rho + d \left\lbrace i \left(\pa g\wedge\omega_\rho -  \opa g \wedge\omega_\rho \right) \right\rbrace + 2i\pa\opa g\wedge\omega_\rho.
\end{eqnarray*}
By (\ref{c}) and Stokes' theorem again (note that $\omega_\rho\mid M =0$), we get
$$\int_{M} g\cdot \ell_\rho 4dv_{M} = \int_\Omega i\pa\opa g\wedge\omega_\rho - \int_\Omega g i\pa\opa\omega_\rho.$$
\end{proof}

\subsection{An estimate for sections over pseudoconcave domains}
\label{subsect:pseudoconcave}

Let $L$ be a holomorphic line bundle on $X$ with a hermitian metric $h$. 
For an (open or closed) subset $W \subset X$ we will use the following notations 
for sections over $W$:
\begin{itemize}
\item $\mathcal{C}^\infty_p (W, L)$ denotes the space of smooth $p$-forms with values in $L$. 
\item $\mathcal{C}^\infty_c(W, L)$ denotes the space of smooth sections of $L$ with compact support in $W$.
\item $L^2_{p,q}(W,L)$ denotes the Hilbert space of square-integrable $(p,q)$-forms with values in $L$ 
with respect to the $L^2$-norm $\Vert\cdot\Vert$ induced by $g$ and $h$.
\end{itemize}

The Chern connection $D= D^\prime+ D^{\prime\prime}$ of $L$ is the unique connection 
whose $(0,1)$-part $D^{\prime\prime}$ coincides with the canonical $\opa$-operator of $L$ and 
which is compatible with the hermitian structure of $L$; that is 
for every $s_1\in \mathcal{C}^\infty_p(X,L), s_2\in\mathcal{C}^\infty_q(X,L)$ we have
\begin{equation}  \label{compatible}
d\lbrace s_1,s_2\rbrace = \lbrace Ds_1,s_2\rbrace + (-1)^p \lbrace s_1, Ds_2\rbrace.
\end{equation}
Here the sequilinear map $\lbrace \ , \  \rbrace$ is defined as usual: 
If $e$ is a local frame of $L$, and $s_1 = f_1\otimes e,\ s_2= f_2\otimes e$, then
$$\lbrace s_1,s_2\rbrace = f_1\wedge\ol f_2 |e|^2.$$
We denote by $\Theta(L)$ the curvature of $D$ for given $h$. \\

We can now prove the following:

\begin{Proposition}
\label{sect}
Assume that $M$ is pseudoconcave, i.e. $\ell_\rho \leq 0$. Then
for any section $s\in\mathcal{C}^\infty(\ol \Omega,L)$ one has
\begin{equation*}   
\label{apriori}
\int_\Omega \vert s\vert^2 \left(-i\Theta(L)\wedge\omega_\rho  - i\pa\opa\omega_\rho \right)
\leq   4\int_\Omega (\vert\opa s\vert^2\cdot\vert\omega_\rho\vert +\vert \opa s\vert\cdot\vert s\vert\cdot\vert\pa\omega_\rho\vert) dV_\omega.
\end{equation*}
\end{Proposition}

\begin{proof}
Let $s\in\mathcal{C}^\infty(\ol \Omega,L)$. We will apply Theorem \ref{integral} with $g= \vert s\vert^2$. \\

First recall that in a local holomorphic frame $e$ one has
$\vert s\vert^2 = \vert f \vert^2 e^{-\psi}$, $D^\prime = \pa -\pa\psi\wedge\cdot$ and $i\Theta(L) = i\pa\opa\psi$
where we denote $s= f \otimes e$ and $|e|^2 = e^{-\psi}$.  

A straightforward computation shows that
\begin{eqnarray*}
i\pa\opa\vert s\vert^2 & = & i\big( \ol{\opa f}\wedge\opa f + D^\prime f \wedge\ol{D^\prime f}- \pa\opa\psi \vert f\vert^2 \big) e^{-\psi}\\
& & +i \big( (\pa\opa f -\pa\psi\wedge\opa f)\ol{f} + f(\pa\opa\ol{f} -\pa\ol{f}\wedge\opa\psi)\big) e^{-\psi}.
\end{eqnarray*}
Hence we obtain
\begin{eqnarray} \label{1}
\int_\Omega i\pa\opa\vert s\vert^2\wedge\omega_\rho & = & \int_\Omega ( i\ol{\opa s}\wedge\opa s +  iD^\prime s\wedge\ol{D^\prime s} -\vert s\vert^2 i\Theta(L) ) \wedge\omega_\rho\\
& &  +2\mathrm{Re} \int_\Omega i\lbrace D^\prime\opa s,s\rbrace \wedge\omega_\rho \nonumber\\
& \geq & \int_\Omega -\vert s\vert^2 i\Theta(L)  \wedge\omega_\rho
+2\mathrm{Re} \int_\Omega i\lbrace D^\prime\opa s,s\rbrace \wedge\omega_\rho.\nonumber
\end{eqnarray}

Using (\ref{compatible}) we have
$$\lbrace D^\prime \opa s,s\rbrace = \lbrace D\opa s,s\rbrace = \lbrace \opa s, Ds\rbrace + d\lbrace\opa s,s\rbrace .$$
For bidegree reasons we have $\lbrace \opa s, Ds\rbrace\wedge\omega_\rho = \lbrace\opa s,\opa s\rbrace\wedge\omega_\rho$.
Therefore we obtain
\[
\int_\Omega i\lbrace D^\prime\opa s,s\rbrace \wedge\omega_\rho 
= \int_\Omega i\lbrace \opa s,\opa s\rbrace\wedge\omega_\rho + \int_\Omega id\lbrace\opa s,s\rbrace\wedge\omega_\rho,
\]
and the second term is 
\begin{eqnarray*}
\int_\Omega id\lbrace\opa s,s\rbrace\wedge\omega_\rho
& = & \int_\Omega id\big(\lbrace\opa s,s\rbrace\wedge\omega_\rho\big) +\int_\Omega i\lbrace \opa s,s\rbrace\wedge\pa\omega_\rho\\
& = & \int_{M} i\lbrace\opa s,s\rbrace\wedge\omega_\rho   +\int_\Omega i\lbrace \opa s,s\rbrace\wedge\pa\omega_\rho\\
& = & \int_\Omega i\lbrace \opa s,s\rbrace\wedge\pa\omega_\rho,
\end{eqnarray*}
where the last equality holds since $\omega_\rho \mid M = 0$. Combining this with (\ref{1}) we have
\begin{eqnarray*} 
\int_\Omega i\pa\opa\vert s\vert^2\wedge\omega_\rho 
& \geq & \int_\Omega -\vert s\vert^2 i\Theta(L)  \wedge\omega_\rho
-2\mathrm{Re} \int_\Omega \big( i\lbrace \opa s,\opa s\rbrace\wedge\omega_\rho - i\lbrace \opa s,s\rbrace\wedge\pa\omega_\rho\big) \nonumber\\
& \geq & \int_\Omega -\vert s\vert^2 i\Theta(L)  \wedge\omega_\rho - 4 \int_\Omega\vert\opa s\vert^2 \cdot\vert\omega_\rho\vert dV_\omega - 4\int_\Omega \vert\opa s\vert\cdot\vert s\vert \cdot\vert\pa\omega_\rho\vert dV_\omega,\nonumber\\
\end{eqnarray*}
where the last inequality holds since $dV_\omega = -\frac{1}{2} \omega^1\wedge\ol\omega^1\wedge\omega^2\wedge\ol\omega^2$.

Together with Theorem \ref{integral} we then get from our assumption $\ell_\rho \leq 0$
\begin{eqnarray*}
0 & \geq & \int_{M} \vert s\vert^2 \ell_\rho 4dv_{M} \\
 & = & \int_\Omega i\pa\opa \vert s\vert^2\wedge\omega_\rho -\int_\Omega \vert s\vert^2 i\pa\opa\omega_\rho\\
 & \geq & \int_\Omega -\vert s\vert^2 i\Theta(L)  \wedge\omega_\rho -4 \int_\Omega\vert\opa s\vert^2 \cdot \vert\omega_\rho\vert dV_\omega\\
 & & -4\int_\Omega \vert\opa s\vert\cdot\vert s\vert \cdot\vert\pa\omega_\rho\vert dV_\omega
-\int_\Omega \vert s\vert^2 i\pa\opa\omega_\rho.
\end{eqnarray*}
This proves the proposition.
\end{proof}

\section{Finiteness of the space of holomorphic sections over Levi-flat domains}
\label{sect:finiteness}

In this section, we will apply the estimate of Proposition \ref{sect} to negative holomorphic line bundles 
$L =\mathcal{O}(-m)$, $m > 0$, over a domain $\Omega$ with smooth Levi-flat boundary $M$ in $\C\PP^2$. 
Here, of course, we equip $\C\PP^2$ with the Fubini-Study metric $g_{FS}$
and $L=\mathcal{O}(-m)$ with the hermitian metric induced from $g_{FS}$.

We obtain the following

\begin{Proposition}
Let $\Omega \subset\C\PP^2$ be a domain with $\mathcal{C}^2$-smooth Levi-flat boundary $M$ with $\mathrm{Ric}^M(\xi, \xi) > 2-2m$. 
Then the space $L^2_{0,0}(\Omega,\mathcal{O}(-m)) \cap \Ker \opa$ is finite dimensional.
\end{Proposition}

\begin{proof}
Choose a defining function $\rho$ of $M$ so that it agrees with the signed boundary distance function with respect to the Fubini--Study metric near $M$.  
The proof is based on the observation that there exists $\varepsilon > 0$ such that 
\begin{equation}  \label{est}
-i\Theta(\mathcal{O}(-m))\wedge\omega_\rho -i\pa\opa\omega_\rho \geq \varepsilon dV_{\omega_{FS}}
\end{equation} 
holds on a neighborhood of $M$ if and only if $\mathrm{Ric}^M(\xi, \xi) > 2-2m$. 

Let us now present the details.
First note that the first term is 
\[
-i\Theta(\mathcal{O}(-m))\wedge\omega_\rho = m \omega_{FS}\wedge\omega_\rho.
\]
From Lemma \ref{alpha}, the second term is 
\begin{align*}
-i\pa\opa\omega_\rho 
&= -i\pa\opa(2i\pa\rho\wedge\opa\rho) = 2i\pa\opa\rho\wedge i\pa\opa\rho \\
&= 2i(\alpha \wedge \opa \rho + \pa\rho \wedge \ol{\alpha})\wedge i(\alpha \wedge \opa \rho + \pa\rho \wedge \ol{\alpha}) \\
&= -2i \alpha \wedge \ol{\alpha} \wedge \omega_\rho
\end{align*}
on the points in $M$. Hence, it suffices to compare $m \omega_{FS}$ with $2i\alpha\wedge\ol{\alpha}$ on 
the holomorphic tangent space $T^{1,0}M$ as quadratic forms. 
Using a distinguished coordinate $\varphi(\zeta,t)$ around a point $p \in M$, we have
\[
m\omega_{FS}\left(\base{\zeta}, \base{\ol{\zeta}}\right) = m
\]
and 
\[
2 i\alpha \wedge \ol{\alpha}\left(\base{\zeta}, \base{\ol{\zeta}}\right)
= \frac{1}{2}(2 - \Ric^M(\xi, \xi))
\]
from Proposition \ref{local} and our normalization of the Fubini--Study metric.
Therefore, 
\[
m > \frac{1}{2}(2 - \Ric^M(\xi, \xi)) \iff \Ric^M(\xi, \xi) > 2-2m
\]
confirms the observation $(5.1)$. \\

From (\ref{est}), Proposition \ref{sect} implies that 
if $\Omega^\prime$ is a sufficiently large, relatively open subset of $\Omega$, 
then for $s\in \mathcal{C}^\infty_c(\ol \Omega\setminus \ol \Omega^\prime,\mathcal{O}(-m))$ we have
\begin{eqnarray*}  \label{3}
 \varepsilon \Vert s\Vert^2 & \leq &  4\int_\Omega( \vert\opa s\vert^2\cdot\vert\omega_\rho\vert +\vert \opa s\vert\cdot\vert s\vert\cdot\vert\pa\omega_\rho \vert) dV_{\omega_{FS}} \nonumber\\
& \leq & 4\int_\Omega( \vert\opa s\vert^2\cdot\vert\omega_\rho\vert + \frac{\delta}{2}\vert s\vert^2 + \frac{1}{2\delta} \vert \opa s\vert^2 \cdot\vert\pa\omega_\rho \vert^2) dV_{\omega_{FS}}
\end{eqnarray*} 
for all $\delta > 0$. But this immediately implies that there exists a constant $C_0> 0$ such that 
\begin{equation}  \label{4}
\Vert s\Vert^2 \leq C_0 \Vert\opa s\Vert^2\qquad \mathrm{for\ all}\ s\in \mathcal{C}^\infty_c(\ol \Omega\setminus \ol \Omega^\prime,\mathcal{O}(-m)).
\end{equation}

The estimate (\ref{4}) implies that that there exists a constant $C> 0$ such that 
\begin{equation}   \label{5}
\Vert s\Vert^2 \leq C\Vert\opa s\Vert^2 + C\int_K \vert s\vert^2 dV_{\omega_{FS}}
\qquad \mathrm{for\ all}\ s\in \mathcal{C}^\infty_c(\ol \Omega,\mathcal{O}(-m)),
\end{equation}
where $K \subset \Omega$ is any compact containing $\Omega^\prime$ in its interior. Indeed, let $\chi$ be a smooth function on $\C\PP^2$, $0\leq\chi\leq 1$, which vanishes in a neighborhood of $\ol{\Omega^\prime}$ and equals $1$ in the complement of $K$. 
Applying (\ref{4}) to $\chi s$ we get
\begin{eqnarray*}
\Vert s\Vert^2 & \leq & 2\Vert \chi s\Vert^2 + 2\Vert (1-\chi)s\Vert^2 \\
& \leq & 2C_0\Vert\opa(\chi s)\Vert^2 + 2\int_K \vert s\vert^2 dV_{\omega_{FS}} \\
& \leq & 4C_0 \Vert \opa s\Vert^2 + 4C_0 \Vert s\opa\chi \Vert^2 + 2\int_K \vert s\vert^2 dV_{\omega_{FS}}
\end{eqnarray*}
from which (\ref{5}) follows. It is standard to deduce from estimate (\ref{5}) the finite dimensionality of $L^2_{0,0}(\Omega,\mathcal{O}(-m))\cap\mathrm{Ker}\opa$. 

\end{proof}

Considering the canonical bundle $L = K_{\C\PP^2} = \mathcal{O}(-3)$, i.e., $m = 3$, 
we have the following finiteness result, which enables us to prove our result in the next section.

\begin{Corollary}
\label{finite}  
Let $\Omega \subset\C\PP^2$ be a domain with $\mathcal{C}^2$-smooth Levi-flat boundary $M$ with $\mathrm{Ric}^M(\xi, \xi) > -4$. 
Then the space $L^2_{0,0}(\Omega,\mathcal{O}(-3)) \cap \Ker \opa$ is finite dimensional.
\end{Corollary}

\section{Proof of the main theorem}
\label{sect:infiniteness}

In this section, we will complete the proof of our Main theorem. The idea is to combine Corollary \ref{finite} with the following well known result:\\

{\it Let $X$ be an $n$-dimensional K\"ahler manifold. Assume that $X$ admits a complete K\"ahler metric and a bounded plurisubharmonic function which is strictly plurisubharmonic on a nonempty open subset of X. Then the space of $L^2$-holomorphic $n$-forms is infinite dimensional.}\\

This result is essentially contained in \cite{H} or \cite{Dem}. In our paper, we will need this very general result for the special case of a pseudoconvex domain in $\C\PP^n$:

\begin{Proposition}
\label{separate} 
Let $\Omega\subset\C\PP^n$ be a domain with $\mathcal{C}^2$-smooth pseudoconvex boundary. 
Then  $\mathrm{dim} L^2_{0,0}(\Omega,\mathcal{O}(-n-1))\cap \mathrm{Ker}\opa = +\infty$.
\end{Proposition}

Note that in $\C\PP^n$, holomorphic $n$-forms can be identified with holomorphic sections of $\mathcal{O}(-n-1)$. 
To make this paper self-contained, we include a full proof of this proposition.
We will use the $L^2$-estimates for $\opa$ in the refined version of Demailly \cite{Dem}:\\

{\it Let $(X,\omega)$ be a K\"ahler manifold of dimension $n$. Assume that $X$ is weakly pseudoconvex. Let $E$ be a hermitian holomorphic vector bundle over $X$, and let $\psi\in L^1_{\mathrm{loc}}$ be a weight function. Suppose that
$$i\Theta(E)+ i\pa\opa\psi\otimes\mathrm{Id}_E \geq \gamma\omega\otimes \mathrm{Id}_E$$
for some continuous positive function $\gamma$ on $X$. Then for any $(n,q)$-form $f$ with $L^2_{\mathrm{loc}}$ coefficients, $q\geq 1$, satisfying $\opa f =0$ and $\int_\Omega \gamma^{-1}\vert f\vert^2 e^{-\psi} dV_\omega < +\infty,$ there exists $u\in L^2_{n, q-1}(X,E)$ such that $\opa u = f$ and 
$$\int_\Omega \vert u\vert^2 e^{-\psi} dV_\omega \leq \frac{1}{q} \int_\Omega \gamma^{-1}\vert f\vert^2 e^{-\psi} dV_\omega.$$
}

\begin{proof}[Proof of Proposition \ref{separate}]
It suffices to show that for every $k\in\N$ we have $\dim L^2_{0,0}(\Omega,\mathcal{O}(-n-1))\cap\mathrm{Ker}\opa \geq k$.\\

Choose $k$ distinct points $p_1,\ldots p_k \in\Omega$. For each $j=1,\ldots,k$, we choose a local holomorphic coordinate system $(z_1^j,\ldots, z^j_n)$ around $p_j$. In such a coordinate system, we identify $p_j$ with $z^j= (0,\ldots,0)$, and we also assume that $B^j_{4\varepsilon} = \lbrace z\mid \vert z^j\vert < 4\varepsilon\rbrace\subset\Omega$ are mutually disjoint by taking enough small $\varepsilon > 0$. 
Let us further choose a smooth function $\varphi$ on $\Omega\setminus\bigcup_{j=1}^k\lbrace p_j\rbrace$ such that $\varphi \equiv 0$ outside $\bigcup_{j=1}^k B^j_\varepsilon$ and 
$\varphi (z^j) = n\log\vert z^j\vert^2$ if $\vert z^j\vert < \frac{\varepsilon}{2}$. We then have $i\pa\opa\varphi \geq - C_1 \cdot\omega_{FS}$ for some constant $C_1 > 0$.\\

Next, we will use results of Takeuchi \cite{T} and Ohsawa and Sibony \cite{OS}, namely 
if $\Omega\subset\C\PP^n$ is a pseudoconvex domain with $\mathcal{C}^2$-smooth boundary, then there exists some $\eta\in (0,1)$ such that $\psi = -\delta^\eta$ is strictly plurisubharmonic in $\Omega$; here $\delta$ is the (unsigned) boundary distance function with respect to $\omega_{FS}$. 
But this implies that for some constant $C> 0$, we have $i\pa\opa(\varphi + C\psi) > 0$ in $\Omega$.\\

Finally, we choose  $\mathcal{C}^\infty$-smooth functions $\chi_j$, $j=1,\ldots, k$ with compact support in $\Omega$ such that $\chi_j (z)=1$ if $z\in B^j_{\varepsilon/2}$ and $\chi(z)=0$ if $z\in\Omega\setminus B^j_\varepsilon$.\\

Observe that the $(0,1)$-forms $f_j = \opa (\chi_j + \sum_{j\not= \ell}\chi_\ell z_1^\ell)$ have compact support in  $\bigcup_{j=1}^k \ol{B^j_{2\varepsilon}\setminus B^j_{\varepsilon/2}}$.\\

Now a $(0,1)$-form with values in $\mathcal{O}(-n-1)$ can be naturally identified with an $(n,1)$-form with values in $L= \Lambda^n T\C\PP^n \otimes\mathcal{O}(-n-1) \simeq\mathcal{O}(n+1)\otimes\mathcal{O}(-n-1)\simeq \C$. \\

From \cite{T} we know that every pseudoconvex domain $\Omega\subset\C\PP^n$ is Stein, hence $(\Omega,\omega_{FS})$ satisfies the assumptions of Demailly's $L^2$-existence result. Thus we get  functions $u_j$ satisfying $\opa u_j = f_j$  with the property
$$\int_\Omega \vert u_j\vert^2 e^{-(\varphi + C\psi)} dV_{\omega_{FS}} \leq \int_\Omega \gamma^{-1}\vert f_j\vert^2_{\omega_{FS}} e^{-(\varphi + C\psi)} dV_{\omega_{FS}} < + \infty,$$
where $\gamma >0$ is the smallest eigenvalue of $i\pa\opa(\varphi + C\psi)$ with respect to $\omega_{FS}$.\\

Notice that $e^{-\varphi} = \vert z^j\vert^{-2n}$ for all $z\in B^j_{\varepsilon/2}$ and that $e^{-C\varphi}$ is bounded from above and from below by constants on $\Omega$. Hence we must have that $u_j(p_\ell) =0$ for $\ell = 1,\ldots,k$. Furthermore, we have that $h_j = \chi_j + \sum_{j\not= \ell}\chi_\ell z_1^\ell -u_j $ satisfies $\opa h_j =0$ and 
$$\int_\Omega \vert h_j\vert^2 dV_{\omega_{FS}} < +\infty.$$
By our construction, $h_j(p_j) = 1-0 =1$, and for $\ell \not= j$, $h_j(p_\ell) = 0-0 =0$. So in particular, we have $h_j(p_\ell) = \delta_{j\ell}$.

Hence $\mathrm{dim} L^2_{0,0}(\Omega,\mathcal{O}(-n-1))\cap \mathrm{Ker}\opa \geq k$.
\end{proof}

We would like to remark that Proposition \ref{separate} gives an improvement for the unweighted case of \cite[Proposition  4.3]{HI}.\\

{\bf Proof of the Main theorem.} Assume by contradiction, that $M$ is an oriented $\mathcal{C}^2$-smooth closed Levi-flat real hypersurface in $\C\PP^2$ such that $\mathrm{Ric}^M(\xi,\xi) > -4$ everywhere along $M$. Then $M$ bounds a domain $\Omega$ and, using Corollary \ref{finite}, $L^2_{0,0}(\Omega,\mathcal{O}(-3))\cap\mathrm{Ker}\opa$ is finite dimensional. On the other hand, this space is infinite dimensional according to Proposition \ref{separate}. This contradiction completes the proof. \hfill$\square$\\

\appendix
\section{A revisit to Takeuchi's inequality}
\label{sect:takeuchi}

In this appendix, we revisit Takeuchi's inequality in a special case, domains with Levi-flat boundary, 
and reveal an equality that is hidden behind Takeuchi's inequality. 
Based on this formula, we also observe that recent studies on the Diederich--Fornaess index \cite{AB}, \cite{FS} 
involve a general restriction on the totally real Ricci curvature of Levi-flat real hypersurfaces in K\"ahler surfaces.

Let us recall Takeuchi's inequality with the explicit constant.
\begin{Theorem}[\cite{T}, \cite{GW}]
Let $\Omega \varsubsetneqq \C\PP^n$ $(n \geq 1)$ be a proper pseudoconvex domain. 
Denote by $\delta$ the unsigned boundary distance function to $\pa\Omega$ with respect to 
the Fubini--Study metric $\omega_{FS}$. Then the inequality
\[
i\pa\opa(-\log \delta) \geq \frac{1}{3}\omega_{FS}
\]
holds in $\Omega$ in the sense of currents.
\end{Theorem}

When $n=2$ and the boundary $\pa\Omega$ is a $\mathcal{C}^3$-smooth Levi-flat real hypersurface $M$, 
we can deduce an inequality along $M$. 
\begin{Corollary}
Let $M$ be an oriented $\mathcal{C}^3$-smooth Levi-flat real hypersurface without boundary in $\C\PP^2$. 
Consider the signed boundary distance function to $M$ with respect to $\omega_{FS}$ 
and induce from it a hermitian metric $h$ on the holomorphic normal bundle $N^{1,0}_M$.
Then, the Chern curvature of $h$ along $T^{1,0}M$ denoted by $\Theta(N^{1,0}_M)$ satisfies the inequality
\begin{equation}
\label{takeuchi-on-boundary}
i\Theta(N^{1,0}_M) \geq \frac{1}{3}\omega_{FS}
\end{equation}
as quadratic forms on $T^{1,0}M$. 
\end{Corollary}
\begin{proof}
By taking a limit of Takeuchi's inequality toward the tangential directions of $M$. 
See the proof of \cite[Proposition 3.3]{A}. 
\end{proof}

\begin{Remark}
One can see that the constant in (\ref{takeuchi-on-boundary}) can be improved to $1/2$
if one looks carefully at the proof of Takeuchi's inequality. See e.g. \cite{CS}.
\end{Remark}

Now we are going to derive an adjunction-type equality hidden in the inequality (\ref{takeuchi-on-boundary}). 
Let $M$ be an oriented $\mathcal{C}^3$-smooth Levi-flat real hypersurface without boundary 
in a K\"ahler surface $(X, J_X, g)$. We restrict $g$ on $M$ as before and use the same notation as in \S\ref{sect:prelim} and \S\ref{sect:adjunction}.
We shall consider an extrinsic curvature of the leaves of the Levi foliation $\mathcal{F}$ in $M$: 
we define for $p \in M$
\[
G_{\mathcal{F}/M}(p) := g(\nabla^M_\base{x_1} \xi, \base{x_1})g(\nabla^M_\base{y_1} \xi, \base{y_1}) - g(\nabla^M_\base{x_1} \xi, \base{y_1})^2 
\]
where $z_1 = x_1 + iy_1$ is the first coordinate of a distinguished parametrization 
around $p$, say $\varphi(\zeta,t): V \to M$ in $(z_1 = x_1 + iy_1, z_2)$.
This curvature $G_{\mathcal{F}/M}$ may be referred to as the \emph{Gauss--Kronecker curvature} or the \emph{Lipschitz--Killing curvature}
of the leaves of $\mathcal{F}$. They are usually defined for real hypersurfaces or real submanifolds in the Euclidean spaces 
by the same formula, namely, the determinant of the shape operator.  

We used the shape operator of the leaves of $\mathcal{F}$ in $M$ to define $G_{\mathcal{F}/M}$. 
Instead of it, one may use the shape operator $A$ of $M$ in $X$ because of the K\"ahlerity of $g$: 
we have at $p \in M$
\begin{align}
\label{extrinsic-shape}
G_{\mathcal{F}/M}(p)
&= g(\nabla_\base{x_1} J_X\nu, \base{x_1}) g(\nabla_\base{y_1} J_X\nu, \base{y_1}) \\
&\quad -  g(\nabla_\base{x_1} J_X\nu, \base{y_1}) g(\nabla_\base{y_1} J_X\nu, \base{x_1}) \nonumber\\
&= g(\nabla_\base{x_1} \nu, J_X\base{x_1}) g(\nabla_\base{y_1} \nu, J_X\base{y_1}) \nonumber\\
&\quad -  g(\nabla_\base{x_1} \nu, J_X\base{y_1}) g(\nabla_\base{y_1} \nu, J_X\base{x_1}) \nonumber\\
&= -g(A \base{x_1}, \base{y_1})^2 + g(A \base{x_1}, \base{x_1}) g(A \base{y_1}, \base{y_1}). \nonumber 
\end{align}
Hence $G_{\mathcal{F}/M} \leq 0$ because the leaves of $\mathcal{F}$ are minimal in $X$ and 
the trace of the shape operator $A$ restricted on $\sigma_T$ is zero.

\begin{Proposition}
\label{takeuchi}
The following equality holds on $M$:
\[
H\left(\sigma_T, \sigma_N \right) - 2G_{\mathcal{F}/M} =  4 i\Theta(N^{1,0}_M)/\omega
\]
where the ratio of $i\Theta(N^{1,0}_M)$ and $\omega$ is taken as quadratic forms on $T^{1,0}M$.
\end{Proposition}
\begin{proof}
We work in the same local situation of Proposition \ref{local}.
We denote by $P = \C \times \{ 0 \} \cap V$ the plaque of the leaf passing through $p$ and
by $\Ric^P$ its Ricci tensor with respect to the restriction of $g$ on the leaf. 

Since (\ref{adjunction-metric}) is equivalent to
\[
g\left(\base{z_1}, \base{z_1}\right) h^2_\varphi =  \det \left[ g\left(\base{z_j}, \base{z_k}\right) \right]_{j, k = 1, 2}  , 
\]
we have at $p$
\begin{align*}
i\Theta(N^{1,0}_M)/\omega 
&= \frac{\pa^2}{\pa z_1 \pa \ol{z_1}} \left(-\log h_\varphi \right) \\
&= \frac{1}{2} \left(\Ric(\base{z_1}, \base{\ol{z_1}}) - \Ric^P(\base{z_1}, \base{\ol{z_1}})\right)\\
&= \frac{1}{4} \left( H(\sigma_T, \sigma_N) + H(\sigma_T, \sigma_T) - g(R^P(\base{x_1}, \base{y_1})\base{y_1}, \base{x_1})\right).
\end{align*}
Repeated use of Gauss' equation and (\ref{extrinsic-shape}) yields
\begin{align*}
& H(\sigma_T, \sigma_T) - g(R^P(\base{x_1}, \base{y_1})\base{y_1}, \base{x_1}) \\
&= H(\sigma_T, \sigma_T) - g(R^M(\base{x_1}, \base{y_1})\base{y_1}, \base{x_1}) - G_{\mathcal{F}/M}\\
&= -2G_{\mathcal{F}/M}
\end{align*}
and this completes the proof. 
\end{proof}

\begin{Remark}
By exploiting the formula of Matsumoto \cite{M}, 
Ohsawa \cite{O} substantially derived this equality when $X = \C^2$ and $g$ is the Euclidean metric.
Here we quote the result of Ohsawa:\\

{\it
Let $A$ be a non-singular complex curve in $\C^2$ and denote by $\delta_A$ the Euclidean distance to $A$. 
Take $z_0 \in A$ and a real normal line $\nu$ of $A$ at $z_0$. 
Consider smooth level sets $M_\varepsilon := \delta_A^{-1}(\varepsilon)$ for $0 < \varepsilon \ll 1$.
Then, $i\pa\opa \log \delta_A(z)$ evaluated on $T^{1,0}_z M_{\delta_A(z)}$
tends to the Lipschitz--Killing curvature of $A$ at $z_0$ in the direction of $\nu$
as $z \to z_0$ through the normal line $\nu$.
}

\end{Remark}

We conclude this paper with a remark on a general restriction on the totally real Ricci curvature 
of Levi-flat real hypersurfaces in K\"ahler surfaces. 
In the first part of the proof of the theorem of Bejancu and Deshmukh, 
the following rigidity result is used.

\begin{Proposition}[{\cite[Remark in p.272]{BD}}]
Let $M$ be an oriented $\mathcal{C}^\infty$-smooth compact Levi-flat real hypersurface without boundary
in a K\"ahler manifold of dimension $\geq 2$.
Suppose that the totally real Ricci curvature $\Ric^M(\xi, \xi) \geq 0$ is non-negative along $M$. 
Then $\Ric^M(\xi, \xi) = 0$ everywhere on $M$. 
\end{Proposition}

We would now like to point out that a weaker version of this proposition follows from our adjunction-type equalities, 
Proposition \ref{local} and \ref{takeuchi}, and recent studies on the Diederich--Fornaess index \cite{AB}, \cite{FS}.

\begin{Corollary}
Let $M$ be an oriented $\mathcal{C}^3$-smooth closed Levi-flat real hypersurface in a K\"ahler surface $(X, J_X, g)$.
Suppose that $M$ is a boundary of a relatively compact domain $\Omega \Subset X$. 
Then, the totally real Ricci curvature $\Ric^M(\xi, \xi)$ cannot be $> 0$ everywhere on $M$. 
\end{Corollary}

\begin{proof}[Sketch of the proof]
Suppose $\Ric^M(\xi, \xi) > 0$ on $M$. Proposition \ref{local} and \ref{takeuchi} yield
\begin{align*}
4i\Theta(N^{1,0}_M)/\omega  &= H(\sigma_T, \sigma_N) - 2G_{\mathcal{F}/M}\\
& \geq H(\sigma_T, \sigma_N)\\
&=  4i\alpha\wedge\ol{\alpha}/\omega + \Ric^M(\xi, \xi) \\
& > 4i\alpha\wedge\ol{\alpha}/\omega.
\end{align*}
We therefore have an inequality
$ i\Theta(N^{1,0}_M) > i\alpha\wedge\ol{\alpha}$
as quadratic forms on $T^{1,0}M$.

Following \cite[Theorem 1.1]{A}, this inequality implies that the Diederich--Fornaess exponent in a weak sense of 
the signed boundary distance function $\rho$ to $M$ with respect to $g$ is greater than $1/2$, 
namely, $-\sqrt{|\rho|}$ is strictly plurisubharmonic in $\Omega$ except a compact subset. 
This contradicts the global restriction on the Diederich--Fornaess index of relatively compact domains
with Levi-flat boundary proved in \cite{AB} and \cite{FS}. 

More precisely, since our domain $\Omega$ need not to be Stein, the global restriction stated in these papers cannot be applied literally,
however, the argument in \cite{FS} still works in the current setting. We leave its detail to the reader. 
\end{proof}

\end{document}